\documentclass[a4paper,american,11pt]{amsart}

\usepackage{amsmath}
\usepackage{amssymb}
\usepackage{caption}
\usepackage{subcaption}
\usepackage{amsfonts}
\usepackage{amscd}
\usepackage{wrapfig}
\input xy
\xyoption{all}

\usepackage{listings} 
\usepackage[usenames]{color} 
\usepackage{colortbl}
\definecolor{cobalt}{rgb}{0.0, 0.28, 0.67}

\usepackage{hyperref}
\usepackage{tikz-cd}
\usepackage{tikz}
\usetikzlibrary{positioning,calc}
\usepackage{pgfplots}
\pgfplotsset{compat=1.16}

\usepackage{tikz-qtree}
\usetikzlibrary{patterns}

\newtheorem{definition}{Definition}
\newtheorem{lemma}{Lemma}

\newtheorem{theorem}{Theorem}
\theoremstyle{definition}
\newtheorem{remark}{Remark}

\newcommand\ZZ{\mathbb Z}

\newcommand\CC{\mathbb C}
\newcommand\RR{\mathbb R}

\begin{document}

\author{Nikita Kalinin}

\title[Certain summations over Farey pairs]{Legendre duality for certain summations over the Farey pairs} 
\address{Guangdong Technion Israel Institute of Technology (GTIIT),
241 DaXue Road, Shantou, Guangdong Province 515603, P.R. China,  Technion-Israel Institute of Technology, Haifa, 32000, Haifa District, Israel, nikaanspb@gmail.com}
\maketitle

\begin{abstract}
Each irreducible fraction $p/q>0$ corresponds to a primitive vector $(p,q)\in\mathbb Z^2$ with positive coordinates. Such a vector $(p,q)$ can be uniquely written as the sum of two primitive vectors $(a,b),(c,d)\in\ZZ_{\geq 0}^2$ spanning a parallelogram of oriented area one. 

We present new summation formulas over the set of such parallelograms. These formulas depend explicitly on $a,b,c,d$ and thus define a summation over primitive vectors $(p,q)=(a+c,b+d)$ indirectly. Equivalently, these sums may be interpreted as running over Farey pairs, i.e.  pairs of fractions $0\leq c/d<a/b\leq 1$ satisfying $ad-bc=1$. 

The input for our formulas is the graph of a strictly concave function $g$. The terms are the areas of certain triangles formed by tangents to the graph of $g$.  Several of these formulas for different $g$ yield values involving $\pi$. For $g$ a parabola we recover the classical Mordell-Tornheim series (also called the Witten series). As a nice application we also discuss formulas for continued fractions for an arbitrary real number $\alpha$ involving coefficients of the continued fraction and the differences between the convergents and $\alpha$.

Using Hata's work, we show that the terms in the above formulas are the coefficients of the Legendre transform of $g$ in a certain Schauder basis, allowing us to interpret our formulas as Parseval-type identities. We hope that the Legendre duality sheds new light on Hata's approach.

Raising the terms in the above summation formula to the power $s$ we obtain a function $F_g(s)$. We prove that  for a strictly concave $g$, the function $F_g(s)$ converges for $s>2/3$ and diverges at $s=2/3$.  

\end{abstract}

\section{Introduction}

Among other results, in this article we prove the following identity:

\begin{equation}
\label{eq_1}
4\sum\limits^{+} \left(a \cdot\mathrm{arctan}\left(\frac{a}{b}\right)+ c\cdot \mathrm{arctan}\left(\frac{c}{d}\right)-(a+c)\cdot\mathrm{arctan}\left(\frac{a+c}{b+d}\right)\right)^2=\pi.
\end{equation}

Throughout the paper, the notation $\sum\limits^{+}$ means that the quadruples $(a,b,c,d)$ range over all $4$-tuples with $a,b,c,d\in\ZZ_{\geq 0}$ satisfying $ad-bc=1$.  That is, the sum is taken over all area-one parallelograms with vertices $(0,0),(a,b),(c,d),(a+c,b+d)\in \ZZ_{\geq 0}^2$. Since $ad-bc=1$, each such parallelogram corresponds to a matrix $$\begin{pmatrix} a \ b\\ c \ d \end{pmatrix} \in SL(2,\ZZ)$$ with non-negative entries. We refer to the set of such matrices  as {\it the positive part} of $SL(2,\ZZ)$, denoted by $SL_{+}(2,\ZZ)$.

Understanding ``the positive part'' of $SL(2,\ZZ)$ may bring new information about Farey fractions and primitive vectors in $\ZZ_{\geq 0}^2$. 

For $k\geq 1$, let $\operatorname{Farey}_k$ denote the Farey sequence of order $k$, i.e.,  the set of all irreducible fractions in the lowest terms in $[0,1]$ with denominators at most $k$, arranged in the ascending order.  Each primitive vector $(p,q)$ in $\ZZ^2_{\geq 0}$, distinct from $(1,0),(0,1)$, admits a unique decomposition $(p,q)=(a,b)+(c,d)$ with $$\begin{pmatrix} a \ b\\ c \ d \end{pmatrix} \in SL_{+}(2,\ZZ).$$ If $(p,q)$ is a primitive vector with $p\leq q$, then finding such a decomposition is equivalent to expressing the fraction $p/q$ as the mediant of two consecutive fractions $c/d<a/b$ in $\operatorname{Farey}_{q-1}$ since $$p/q = (a+c)/(b+d).$$ The pair of fractions $(c/d,a/b), ad-bc=1$ is called a {\it Farey pair}. Each Farey pair appears as two consecutive fractions in $\operatorname{Farey}_k$ for a certain $k$. For a historical background and modern developments related to Farey fractions and their connection to Diophantine approximation and Riemann hypothesis, see \cite{cobeli2003haros, kim2025metric,yoshimoto2004abelian}.

Sums of the form $\sum f(b,d)$ over Farey pairs $c/d,a/b\in \operatorname{Farey}_k$ have been studied in many works, including  \cite{gupta1964identity}, \cite{newman1969sums} (\cite{apostol} pp. 110-112), \cite{robertson1968sums},  \cite{hall1970note}, \cite{rieger1981farey},  \cite{kanemitsu1982some}, \cite{ishibashi1984farey}, \cite{rieger1984kreisfigur}, \cite{maxsein1985bemerkung}, \cite{hall1994consecutive},   \cite{hata1995farey}, \cite{aliev1998metric},  \cite{girstmair2010farey}, \cite{cobeli2010intervals},   \cite{chaubey2017correlations} for various functions $f$. Our formulas can also be interpreted as sums over Farey pairs. However, in contrast to previous works where the summands depend only on denominators $b,d$,  our terms also involve numerators $a,c$, leading to new identities.

As a particular example, consider the following function $f(a,b,c,d)$ for a Farey pair $c/d,a/b$:

\begin{equation}
\label{main}
f(a,b,c,d)=\sqrt{a^2+b^2}+\sqrt{c^2+d^2}-\sqrt{(a+c)^2+(b+d)^2},
\end{equation}
measuring the defect in the triangle inequality for the triangle with vertices $(0,0),(a,b),(a+c,b+d)$.

In the past, motivated by tropical sandpile caustics \cite{us_series} (see further exploration of tropical caustics in \cite{mikhalkin2023wave}), together with  M. Shkolnikov we had evaluated the sum of $f$ over $SL_{+}(2,\ZZ)$. We obtained the following result: 
\begin{theorem}[\cite{easter}]\label{thm_circle} The following two formulas hold: 
\begin{align}
&\sum\limits^{+} \left(\sqrt{a^2+b^2}+\sqrt{c^2+d^2}-\sqrt{(a+c)^2+(b+d)^2}\right) = 2, \label{eq_triangle}\\  
&\sum\limits^{+} \left(\sqrt{a^2+b^2}+\sqrt{c^2+d^2}-\sqrt{(a+c)^2+(b+d)^2}\right)^2 = 2-\pi/2.
\end{align}
\end{theorem}
To derive these formulas, we may consider a unit circle in $\RR^2$, then iteratively draw certain tangent lines and sum up the areas of triangles formed by these tangent lines. 
Choosing other initial curves, in this article we get more formulas such as \eqref{eq_1} (the case of cycloid). We present the general construction for an arbitrary  concave function $g$ in Theorem~\ref{thm_short}. Although this result appeared (in different notation) in papers \cite{easter, pi_short}, its proof is recalled in Section~\ref{sec_one}, and yet another proof (via telescoping) is given in Section~\ref{sec_3}.

In this article we explore the applicability of our summation method and investigate the role of Legendre duality. We may summarize the conceptual relations in the following diagram:

\begin{figure}[h]
\centering
\begin{tikzpicture}[>=latex, node distance=4cm]

\node (area) [draw, rectangle, rounded corners, minimum width=3cm, align=center]
  {Our approach\\(area)\\ Section~\ref{sec_one}};
\node (hata) [draw, rectangle, rounded corners, minimum width=3cm, right=of area, align=center]
  {Hata's approach\\(integrals)\\Section~\ref{sec_mixed}};

\coordinate (mid) at ($(area)!0.5!(hata)$);

\node (telescoping) [draw, rectangle, rounded corners, minimum width=3cm, below=3cm of mid, align=center]
  {Telescoping\\ Section~\ref{sec_3}};

\draw[<->] (area) -- node[midway, above, sloped] {Legendre duality} (hata);
\draw[->] (area) -- (telescoping);
\draw[<-] (hata) -- (telescoping);

\end{tikzpicture}
\caption{Relation between our approach, Hata's approach, and telescoping.}
\end{figure}

In Section~\ref{sec_two} we derive  formulas similar to Theorem~\ref{thm_circle} for parabolas, and for various other functions (including the cycloid, \eqref{eq_1}) in Section~\ref{sec_others}. Section~\ref{continued} contains several proofs of error-sum-like identities for continued fractions (by carving, by induction, and by telescoping). Raising the terms in the above summation formula to the power $s$ we obtain a function $F_g(s)$. In Section~\ref{sec_3} we use the Legendre dual function to simplify the estimates and prove convergence of $F_g(s)$ for $s>2/3$ as well as to provide an alternative proof of Theorem~\ref{thm_short}. Section~\ref{sec_mixed} reveals the connection between our formulas and Hata's results, we also prove a  theorem about mixed area inspired by Hata's example for the case of Euler--Mascheroni constant $\gamma$. 

In Section~\ref{sec_disc} we discuss the relations of the presented material with the reduced matrices and Selberg trace formula, and mention several yet unexplored future directions.

\tableofcontents
 
 \section{Formulas for the area and $SL(2,\ZZ)$-length}
 \label{sec_one}
 
 In this section, we define the terms $f_g(a,b,c,d)$ and recall the geometric meaning of $F_g(1)$ and $F_g(2)$ for the function $F_g(s)=\sum\limits^{+} f_g(a,b,c,d)^s$, see Figure~\ref{fig_tangents}.
 
Consider a graph of a strictly concave continuous  function $g:[x_0,x_1]\to\RR$, i.e. $$g\left(\frac{x+y}{2}\right)< \frac{g(x)+g(y)}{2}\text{ for } x,y\in[x_0,x_1], x\ne y.$$ Suppose that $g$ is differentiable on $[x_0,x_1)$,  $g'(x_0)=0$ and the tangent line to the graph of $g$ at $x_1$ is vertical ($g'(x_1^-)=-\infty$), see Figure~\ref{fig_tangents} for an illustration.  
 
 \begin{definition}\label{def_gtriangle}A $g$-triangle is the curvilinear triangle with three vertices $(x_0,g(x_0)),(x_1,g(x_0)),(x_1,g(x_1))$ where $(x_0,g(x_0)),(x_1,g(x_1))$ are connected by the graph of $g$ and two other sides are horizontal and vertical straight intervals.
 \end{definition}

Given a $g$-triangle, each primitive vector $(a,b)\in\ZZ_{\geq 0}^2$, determines a tangent line $L_{a,b}$ for the graph of $g$, such that  $L_{a,b}$ is orthogonal to the vector $(a,b)$ and intersects the two straight sides of this $g$-triangle. Let the equation of $L_{a,b}$ be $$ax+by-\gamma_{a,b}=0.$$ At the point $(x_{ab},g(x_{ab}))$ where $L_{a,b}$ touches the graph of $g$ (this point is unique since $g$ is strictly concave), $L_{a,b}$ has the following form: 
\begin{equation}\label{eq_legendrian}y(x) = -\frac{a}{b}\,x+\frac{\gamma_{a,b}}{b},\end{equation} so $g'(x_{ab}) = \frac{-a}{b}$.

  \begin{figure}[h]
\begin{center}
\begin{tikzpicture}[scale=2]
\draw (1,0) arc (0:100:1);
\draw (1,0)--(1,1.3);
\draw (0,1)--(1.3,1);
\draw (0,1.4142)--(1.4142,0); 
\fill[gray] (0.4142,1)--(1,1)--(1,0.4142)--cycle;
\path[pattern=north east lines]  (0,1)--(1,1)--(1,0) arc (0:100:1)--cycle;

\draw (1,0) node{$\bullet$};
\draw (0,1) node {$\bullet$};
\draw (1,1) node {$\bullet$};
\draw (1/1.4142,1/1.4142) node {$\bullet$};

\draw (1,0) node[left] {$C$};
\draw (0,1) node[above] {$A$};
\draw (1,1) node[above right] {$B$};

\draw (-0.3,-0.3)--(1.3,-0.3);
\draw (0,-0.3) node{$\bullet$};
\draw (1,-0.3) node{$\bullet$};

\draw (1,-0.5) node {$x_1$};
\draw (0,-0.5) node {$x_0$};


\end{tikzpicture}
\end{center}
\caption{The graph of $g$ is the curve from $A$ to $C$. $ABC$ is a $g$-triangle. The gray triangle is the triangle $\Delta_{a,b,c,d}$ for $(a,b)=(1,0),(c,d)=(0,1), (a+c,b+d)=(1,1)$. The area of $\Delta_{a,b,c,d}$ is $\frac{1}{2}f_g^2(a,b,c,d)$. The area of the $g$-triangle, i.e., the region which is filled with parallel lines, is $\frac{1}{2}F_g(2)$, while $F_g(1)$ is $|AB|+|BC|$.}
\label{fig_tangents}
\end{figure}

Given $g$, define $f_g(a,b,c,d)$  in the following way: for two primitive vectors $(a,b),(c,d)\in\ZZ_{\geq 0}^2$ with $ad-bc=1$ draw the tangent lines $L_{a,b},L_{c,d},L_{a+c,b+d}$, given by equations

\begin{align*}
&ax+by-\gamma_{a,b}=0 \\
&cx+dy-\gamma_{c,d}=0 \\
&(a+c)x+(b+d)y-\gamma_{a+c,b+d}=0.\\
\end{align*}
The graph of $g$ lies in the intersection of the half-planes where these functions are non-positive.

Define $$f_g(a,b,c,d)=\sqrt{2S_{a,b,c,d}},$$ where $S_{a,b,c,d}$ is the area of the triangle bounded by $L_{a,b},L_{c,d},L_{a+c,b+d}$. Denote this triangle by $\Delta_{a,b,c,d}$:

\begin{align*}
&ax+by-\gamma_{a,b}\leq 0 \\
&cx+dy-\gamma_{c,d}\leq 0 \\
&(a+c)x+(b+d)y-\gamma_{a+c,b+d}\geq 0.\\
\end{align*}

The vertices of this triangle are

\begin{align*}
(x_{12},y_{12})=&(d\gamma_{a,b}-b\gamma_{c,d}, a\gamma_{c,d}-c\gamma_{a,b})\\
(x_{13},y_{13})=&((b+d)\gamma_{a,b}-b\gamma_{a+c,b+d}, a\gamma_{a+c,b+d}-(a+c)\gamma_{a,b})  \\
(x_{23},y_{23})=&(d\gamma_{a+c,b+d}-(b+d)\gamma_{c,d}, (a+c)\gamma_{c,d}-c\gamma_{a+c,b+d}).\\
\end{align*}

So $2S_{a,b,c,d}$ is $$|((x_{13}-x_{12})(y_{23}-y_{12})-(x_{23}-x_{12})(y_{13}-y_{12}))|=|(bc-ad)|(\gamma_{a,b}+\gamma_{c,d}-\gamma_{a+c,b+d})^2.$$
Since $ad-bc=1$ we have
\begin{equation}
\label{formula}
f_g(a,b,c,d)=\sqrt{2S_{a,b,c,d}}=\gamma_{a,b}+\gamma_{c,d}-\gamma_{a+c,b+d}.
\end{equation}

The terms $f(a,b,c,d)$ in \eqref{main} in Theorem~\ref{thm_circle} correspond to the choice $g(x)=\sqrt{1-x^2}$, i.e., when the graph of $g$ is the upper arc of the unit circle centered at the origin. Indeed, for this circle the tangent line $L_{a,b}$ is given by $$ax+by-\sqrt{a^2+b^2}=0$$ so $\gamma_{a,b} = \sqrt{a^2+b^2}$ which explains the definition of $f(a,b,c,d)$ in \eqref{main} via \eqref{formula}.

\begin{definition}[Zeta function for a $g$-triangle]\label{def_zeta}
Define $$F_g(s) = \sum\limits^+(f_g(a,b,c,d))^s.$$
\end{definition}
The name `zeta-function' is suggested by Mikhail Shkolnikov. As shown in  \cite{easter}, when $g=\sqrt{1-x^2}$ (the unit circle arc), we have $F_{\sqrt{1-x^2}}(1) = 2$ and $F_{\sqrt{1-x^2}}(2) = 2-\pi/2$. Moreover, the series $F_{\sqrt{1-x^2}}(s)$ converges if $s>2/3$ and diverges at $s=2/3$. For an arbitrary strictly concave $g$ the function $F_g(s)$ has the same behavior at $s=2/3$ as we prove in Theorem~\ref{thm_convergence}.

\begin{theorem}[\cite{pi_short}]
\label{thm_short}
Assume $g$ is strictly concave, continuous on $[x_0,x_1]$, $C^1$ on $[x_0,x_1)$ and $g'(x_1^-)=-\infty$, then  
\begin{equation}
\label{eq_thm1}
F_g(1)=\sum\limits^{+} |\gamma_{a,b}+\gamma_{c,d}-\gamma_{a+c,b+d}|\end{equation}
 equals the sum of the lengths of the two straight sides of the $g$-triangle (i.e., $|AB|+|BC|$ in Figure~\ref{fig_tangents}), while \begin{equation}
\label{eq_thm2}
F_g(2)=\sum\limits^{+} (\gamma_{a,b}+\gamma_{c,d}-\gamma_{a+c,b+d})^2\end{equation} equals twice the area of the $g$-triangle.
\end{theorem}

\begin{proof}
The proof of this theorem is the following geometric cut-and-paste computation. The triangles $\Delta_{a,b,c,d}$, corresponding to terms in $F_g(2)$ tile the $g$-triangle, i.e. they fill  the region between the graph of $g$ and the tangents $L_{1,0}, L_{0,1}$, see Figure~\ref{fig_tangents}.  Since each term in $F_g(2)$ equals twice the area of the corresponding triangle $\Delta_{a,b,c,d}$,  we conclude that $F_g(2)$ equals twice the area of the $g$-triangle.

The geometric interpretation of  $F_g(1)$ is as follows. Consider $SL(2,\ZZ)$-invariant length of the straight intervals of rational slope; sometimes it is called the {\it lattice length}. 
\begin{definition}\label{def_sl2zlength} The $SL(2,\ZZ)$-length (or the lattice length) of an interval $I$ of rational slope is equal to the usual Euclidean length of $I$ divided by the Euclidean length of the primitive vector in the direction of $I$. 
\end{definition}
This notion of length is invariant with respect to the action of $SL(2,\ZZ)$ and parallel translations. Note that each of the triangles $\Delta_{a,b,c,d}$ may be brought to the triangle with vertices $(0,0),(\mu,0),(0,\mu),\mu\geq 0$ by the $SL(2,\ZZ)$ action followed by a translation, so $\Delta_{a,b,c,d}$ is $SL(2,\ZZ)$-equilateral triangle with the sides of $SL(2,\ZZ)$-length $\mu$ equal to $f_g(a,b,c,d)$. So, we start with $SL(2,\ZZ)$-length of the polyline $AB+BC$ which is equal to $|AB|+|BC|$ since the primitive vectors in the horizontal and vertical directions have length one. Then, each time a triangle  $\Delta_{a,b,c,d}$ is carved, the polyline gains one more side and its $SL(2,\ZZ)$-length decreases by $f_g(a,b,c,d)$. Then, if a concave polyline tends to a strictly concave curve, then the sequence of  $SL(2,\ZZ)$-lengths of polylines tends to zero, because the Euclidean lengths of polylines is bounded; and a part of it, which tends to zero, is divided by short primitive vectors, and other part is divided by long primitive vectors (see Definition~\ref{def_sl2zlength}); hence it also tends to zero.  Therefore, $F_g(1)$ is equal to the $SL(2,\ZZ)$-perimeter $|AB|+|BC|$ that we started with.
\end{proof}

In a sense, this method of computing the area is dual to Archimedes' approach for the area under a parabola: Archimedes filled the area  beneath the parabola using inscribed triangles with vertices on the parabola,  while here we carve triangles out of the complement to the curve, and the lines containing the sides of triangles are tangent to the graph of $g$. 

\begin{remark}
Note that our construction is $SL(2,\ZZ)$-invariant. So, instead of a $g$-triangle how it was defined above, for each $ \begin{pmatrix} a \ b\\ c \ d \end{pmatrix} \in SL_{+}(2,\ZZ)$, we may begin with a similar curvilinear triangle whose two straight sides belong to lines $ax+by+\gamma_{a,b}=0, cx+dy+\gamma_{c,d}=0$, then we cut a triangle by the tangent line $(a+c)x+(b+d)y+\gamma_{a+c,b+d}=0$, etc. \end{remark}

Actually, for our construction we only need any piece of strictly concave curve $\Gamma:[0,1]\to\RR^2$ such that the triangle formed by tangents to $\Gamma$ at $\Gamma(0), \Gamma(1)$  and the interval $\Gamma(0)\Gamma(1)$  contains $\Gamma$.

\begin{definition}
For a strictly concave curve $\Gamma:[0,1]\to\RR^2$, the triangle with vertices $\Gamma(0),\Gamma(1)$, and the intersection $C$ of the tangents to $\Gamma$ at $\Gamma(0)=A$ and $\Gamma(1)=B$ is called the {\it support triangle} of $\Gamma$, if it contains $\Gamma$. The curvilinear triangle with sides $AB, BC$ and $\Gamma$ is called a $\Gamma$-triangle.
\end{definition}

 Then we choose equations $\alpha x+\beta y-\gamma_{1,0}=0, \alpha' x+\beta' y-\gamma_{0,1}=0$ of tangents at $A, B$ such that $\det \begin{pmatrix} \alpha \ \beta\\ \alpha' \ \beta' \end{pmatrix}=1$, and $\Gamma$ belongs to the intersection of the halfplanes where these equations are non-positive.

\begin{definition}
Given the above choice, define $F_\Gamma(s) = \sum\limits^+ f_\Gamma(a,b,c,d)^s$ where  $f_\Gamma(a,b,c,d)$ is $\sqrt{2S_{a,b,c,d}}$ and $S_{a,b,c,d}$ is the area of the triangle formed by three lines tangent to $\Gamma$, given by
\begin{align*}
&(a\alpha+b\alpha')x+(a\beta+b\beta')y-\gamma_{a,b}\leq 0 \\
&(c\alpha+d\alpha')x+(c\beta+ d\beta')y-\gamma_{c,d}\leq 0 \\
&((a+c)\alpha + (b+d)\alpha')x+((a+c)\beta+(b+d)\beta')-\gamma_{a+c,b+d}\geq 0.\\
\end{align*}
\end{definition}

This definition is equivalent to Definition~\ref{def_zeta} for $\alpha=\beta'=1, \alpha'=\beta=0$.  Define the $SL(2,\ZZ)$-length of $AB$ as its Euclidean length $|AB|$ divided by $\sqrt{\alpha^2+\beta^2}$ and the $SL(2,\ZZ)$-length of $BC$ as its Euclidean length $|BC|$ divided by $\sqrt{\alpha'^2+\beta'^2}$. Then we have the following theorem, whose proof is identical to that of Theorem~\ref{thm_short}.

\begin{theorem}
\label{thm_short2}
Consider a $\Gamma$-triangle, then 
\begin{equation}
\label{eq_thm11}
F_\Gamma(1)=\sum\limits^{+} |\gamma_{a,b}+\gamma_{c,d}-\gamma_{a+c,b+d}|\end{equation}
 equals the sum of the $SL(2,\ZZ)$ lengths of the two straight sides of the $\Gamma$-triangle, while \begin{equation}
\label{eq_thm22}
F_\Gamma(2)=\sum\limits^{+} (\gamma_{a,b}+\gamma_{c,d}-\gamma_{a+c,b+d})^2\end{equation} equals twice the area of the $\Gamma$-triangle.
\end{theorem}

\section{Mordell--Tornheim series}
\label{sec_two}
In this section we study in detail the formulas for parabolas that can be obtained using our method. The computations for other classical curves are presented in Section~\ref{sec_others}.

Consider the parabola $$y=1-(x-y)^2,$$ and take the portion of its graph between the horizontal and vertical tangent lines. This portion corresponds to the graph of the function $$y=g(x) = \frac{2x-1+\sqrt{5-4x}}{2}.$$

 By direct computation, the tangent line $ax+by-\gamma_{a,b}=0$ for this $g$ has  $\gamma_{a,b} = \frac{a^2}{4(a+b)}+a+b$ and
$$\gamma_{a,b}+\gamma_{c,d}-\gamma_{a+c,b+d} =  \frac{1}{4(a+c)(b+d)(a+b+c+d)}.$$

Hence, by Theorem~\ref{thm_short}
\begin{align*}
\sum^{+} (\gamma_{a,b}+\gamma_{c,d}-\gamma_{a+c,b+d})^2=\sum^{+} \frac{1}{16(a+c)^2(b+d)^2(a+b+c+d)^2} = \\ = 2( 5/4\cdot 1/4 - \int_{3/4}^1 (y+\sqrt{1-y})dy)=1/48.\end{align*}
$$\sum^{+} \frac{1}{4(a+c)(b+d)(a+b+c+d)} = 5/4-1+ 1-3/4= 1/2.$$

Note that the vector $(a+c,b+d) = (m,n)$ can be any primitive vector in the first quadrant except $(1,0)$ and $(0,1)$, so this expression can also be written as $$\sum^{+} \frac{1}{(a+c)^2(b+d)^2(a+b+c+d)^2}   = \sum\limits_{\substack{(m,n)=1,\\ m,n>0}} \frac{1}{m^2n^2(m+n)^2} = 1/3,$$
and $$\sum\limits_{\substack{(m,n)=1,\\ m,n>0}} \frac{1}{mn(m+n)}=2.$$

The series $\sum\limits_{m,n>0} \frac{1}{m^{s_1}n^{s_2}(m+n)^{s_3}}=T(s_1,s_2,s_3)$ is called a Mordell-Tornheim series, for Tornheim considered it in  \cite{tornheim1950harmonic} to find relations between multiple zeta values, and Mordell evaluated  $T(s_1,s_2,s_3)$ for $s_1=s_2=s_3$ being even integer in \cite{mordell1958evaluation}. 

The series $\zeta_{\mathfrak{su}(3)}(s)=2^s\sum\limits_{m,n>0} \frac{1}{m^sn^s(m+n)^s}$ is also called the Witten series, as Witten defined a zeta series for a compact semisimple Lie algebra $G$ as the sum of $s$-th powers of the dimensions of the irreducible representations of $G$, in this example $\frac{1}{2}mn(m+n)$ being the dimension of an irreducible representation of $\mathfrak{su}(3)$ with the highest weights $m-1$ and $n-1$. 

Mordell--Tornheim series are also related to toroidal $b$-divisors \cite{botero2022toroidal} where the value of $\sum\frac{1}{m^2n^2(m+n)^2}$ represents the self-intersection of a canonical divisor of a toroidal (i.e. blown up infinitely many times $\mathbb P^2$) two-fold.

The Mordell-Tornheim series appears in the study of linear relations between multiple zeta values,  \cite{zagier1994values}, \cite{gangl2006double}. It converges whenever $\mathrm{Re}\ s> 2/3$ \cite{matsumoto2003mordell},\cite{matsumoto2002analytic}. The Mordell-Tornheim series can be presented (for different values of $s$) as a certain integral \cite{komori2008integral},\cite{onodera2012functional}. Analyticity of the generalizations of Mordell--Tornheim series for $s\in\CC$ outside a certain number of explicitly described hyperplanes is proven in \cite{miyagawa2016analytic}.  The residues at $s=2/3, s=1/2-k, k\in\ZZ_{\geq 0}$ are explicitly evaluated in \cite{romik2015number}. Therefore, in this case $F_g(s)$ is an analytic function on $\CC$ with known set of poles. It would be interesting to connect the estimates of lattice point counting under a parabola to zeros of  $F_g(s)$.

For completeness, we show 
\begin{lemma}\label{lemma_parab}
The function
$$\omega(s)=\sum\limits_{m,n\geq 1} \frac{1}{m^sn^s(m+n)^s}$$
converges for $s>2/3$ and diverges at $s=2/3$.
\end{lemma}

\begin{proof}
Consider the half of the sum where $m\geq n$. Fix $m$. Then, for $2/3\leq s <1$ we have $$\sum\limits_{n\leq m} \frac{1}{n^s(m+n)^s}\approx \int_1^m \frac{1}{x^s(m+x)^s}\approx m^{1-2s}\int_0^1\frac{1}{t^s(1+t)^s}.$$

Therefore, $$\sum\limits_{m\geq n\geq 1} \frac{1}{m^sn^s(m+n)^s} \approx \sum_{m=1}^\infty m^{1-3s},$$
which converges for $s>2/3$ and diverges at $s=2/3$.
\end{proof}

\section{Hyperbola, cycloid, tractrix, and astroid}
In this section, we describe the summation formulas obtained from different classical choices of the function $g$, such as  hyperbola, cycloid, tractrix, and astroid.
 
\label{sec_others} Consider the portion of the curve $y^2-(x-2y)^2=1$ between the points where it has a horizontal tangent line and a vertical tangent line, let this serve as a function $g$.  In this case,  $\gamma_{a,b} = -\sqrt{(2a+b)^2-a^2}$ and  the associated series is
 
\begin{align*}\sum^{+} (\sqrt{(2a+b)^2-a^2}+\sqrt{(2c+d)^2-c^2}-\sqrt{(2a+2c+b+d)^2-(a+c)^2})^2=\\
=-2\int_{-2/\sqrt 3}^{-1}(2y+\sqrt{y^2-1}+\sqrt{3}) =\frac{1}{2}\ln 3+2\sqrt{3} -4.
\end{align*}
This can also be rewritten as
$$\sum \left(\sqrt{a^2-b^2}+\sqrt{c^2-d^2}-\sqrt{(a+c)^2-(b+d)^2}\right)$$
over an appropriate subset of $SL(2,\ZZ)$ ($a\geq2b, c\geq 2d$), which can be regarded as the hyperbolic analog of \eqref{main}.

\vskip 1.5ex 
\noindent {\bf Cycloid.}  Consider the segment of cycloid curve defined parametrically by $(x(t),y(t))=(t-\sin t, 1-\cos t)$ between points where the tangent is horizontal and vertical. Then $$\gamma_{a,b} = a\cdot \mathrm{arccos}\frac{a^2-b^2}{a^2+b^2}+2b$$ and the associated summation develops as

$$\sum^{+} \left(a\cdot\mathrm{arccos}\frac{a^2-b^2}{a^2+b^2}+ c\cdot \mathrm{arccos}\frac{c^2-d^2}{c^2+d^2}-(a+c)\cdot\mathrm{arccos}\frac{(a+c)^2-(b+d)^2}{(a+c)^2+(b+d)^2}\right)^2=\pi.$$

Here, we explicitly computed the area of the corresponding $g$-triangle. Equivalently, 
$$4\sum^{+} \left(a \cdot\mathrm{arctan}(\frac{a}{b})+ c\cdot \mathrm{arctan}(\frac{c}{d})-(a+c)\cdot\mathrm{arctan}(\frac{a+c}{b+d})\right)^2=\pi.$$

Note that $\mathrm{arctan}(\frac{a}{b})$ measures the angle between the vector $(a,b)$ and $y$-axis; hence this identity may be viewed as a ``weighted angle version'' of \eqref{main}.

\vskip 1.5ex 
\noindent {\bf Tractrix.} Consider the tractrix curve, $$y(x)= \sqrt{1-x^2}-\ln\frac{1+\sqrt{1-x^2}}{x}$$ defined on the interval  $[-1,0]$. Then $\gamma_{a,b} =b\ln(\sqrt{a^2+b^2})-b\ln b$ and the corresponding sum becomes

$$\sum^{+} \left(\ln(\frac{\sqrt{a^2+b^2}^b\sqrt{c^2+d^2}^d}{\sqrt{(a+c)^2+(b+d)^2}^{b+d}}) +\ln \frac{(b+d)^{b+d}}{b^bd^d}\right)^2=\pi.$$

\vskip 1.5ex 
\noindent {\bf Astroid.} For the curve $x^{2/3}+y^{2/3}=1$ we obtain $\gamma_{ab} =\frac{ab}{\sqrt{a^2+b^2}} $ and so $$\sum^{+} \left(\frac{ab}{\sqrt{a^2+b^2}}+\frac{cd}{\sqrt{c^2+d^2}}-\frac{(a+c)(b+d)}{\sqrt{(a+c)^2+(b+d)^2}}\right)^2 = 3\pi/16.$$

$$\sum^{+} \left(\frac{ab}{\sqrt{a^2+b^2}}+\frac{cd}{\sqrt{c^2+d^2}}-\frac{(a+c)(b+d)}{\sqrt{(a+c)^2+(b+d)^2}}\right) = -2.$$

\section{Identities for error sums of continued fractions} 
\label{continued}

The topic of continued fractions is rich and classical, but it seems that we found certain new and simple identities that might be included as exercises in textbooks. 

Let $\alpha$ be an irrational real number. Consider a simple continued fraction

$$a_{0}\in \mathbb{Z} , a_{n}\in \mathbb{Z}_{\geq 0}, \quad \alpha=a_0+\cfrac{1}{a_1+\cfrac{1}{a_2+\cfrac{1}{a_3+\ldots}}}=\left[a_0, a_1, \ldots\right] $$

Define 

\begin{align}
& h_{-2}=0 \quad h_{-1}=1 \quad h_n=a_n h_{n-1}+h_{n-2}\label{eq_continuous}\\
& k_{-2}=1 \quad  k_{-1}=0\quad  k_n=a_n k_{n-1}+k_{n-2}.
\end{align}

Then $\quad a_0+\cfrac{1}{a_1+\cfrac{1}{a_2+\cfrac{1}{a_3+\ldots +\cfrac{1}{a_n}}}}=\left[a_0, a_1, \ldots, a_n\right] =\frac{h_n}{k_n}$ and $\frac{h_n}{k_n} \rightarrow \alpha.$

\medskip

The following identity is stated in \cite{easter} without a proof (and with a typo), here we will provide three proofs of it.

\begin{theorem}
\label{th1}
If $\alpha$ is an irrational number, then the following identities hold:

\[ \boxed{\;\displaystyle a) \sum_{n=-1}^\infty a_{n+1}|h_n-\alpha k_n| = \alpha+1,}\qquad 
\boxed{\;\displaystyle b) \sum_{n=-1}^\infty a_{n+1}(h_n-\alpha k_n)^2 = \alpha.} \]

If $\alpha=p/q = h_N/k_N$ with $\gcd(p,q)=1$, then

\[\boxed{\;\displaystyle c) \sum_{n=-1}^{N-1} a_{n+1}|h_n-\alpha k_n| = \alpha+1-1/q,} \qquad 
\boxed{\;\displaystyle d) \sum_{n=-1}^{N-1} a_{n+1}(h_n-\alpha k_n)^2 = \alpha.}\]

\end{theorem}

\begin{proof}[First proof. Carving]
 Note that the proof of Theorem~\ref{thm_short} relies on support lines; we do not actually require tangency since the construction involves cutting triangles and computing changes in area and $SL(2,\ZZ)$ perimeter. We only require that for every $(a,b)$, there exists a unique support line with equation $ax+by-\gamma_{a,b}=0$ and then employ the numbers $\gamma_{a,b}$.  
 Let us apply this procedure to the triangle $\Delta_\alpha$ with vertices $(0,0),(1,0),(0,\alpha)$ where $\alpha>0$ is an irrational number. In this case the $g$-triangle is the triangle $DCB$ in Figure~\ref{fig_tr}.

\begin{center}
\begin{figure}[h]
\begin{tikzpicture}[scale=1]
  \pgfmathsetmacro{\a}{sqrt(7)}
  \pgfmathsetmacro{\eps}{0.1}

  \coordinate (A) at (0,0);
  \coordinate (B) at (1,0);
  \coordinate (D) at (0,\a);
  \coordinate (C) at (1,\a);
  \coordinate (C1) at (1,{\a - 1});
  \coordinate (C2) at (1,{\a - 2});

  \pgfmathsetmacro{\t}{3 - \a}
  \coordinate (B1) at ({\t}, {\a - 2*\t});

  \pgfmathsetmacro{\xstart}{1}
  \pgfmathsetmacro{\ystart}{- \eps}

  \pgfmathsetmacro{\xend}{0}
  \pgfmathsetmacro{\yend}{\a + \eps}

  \draw[very thick] (A) -- (B) -- (D) -- cycle;

  \draw[dotted] (D) -- (C) -- (B);

  \draw[thin] (D) -- (C1);
  \draw[thin] (D) -- (C2);

  \pgfmathsetmacro{\xstart}{(3 - \ystart)/3}
  \pgfmathsetmacro{\xend}{(3 - \yend)/3}
  \draw[dashed] (\xstart,\ystart) -- (\xend,\yend);

  \node[below left] at (A) {A};
  \node[below right] at (B) {B};
  \node[above right] at (C) {C};
  \node[above left] at (D) {D};
  \node[right] at (C1) {$C_1$};
  \node[right] at (C2) {$C_2$};

  \fill (B1) circle (1pt);
  \node[right] at (B1) {$B_1$};
\end{tikzpicture}
\caption{The first few steps of our procedure consists of cutting $DCC_1$, then $DC_1C_2$, then $BC_2B_1$.}
\label{fig_tr}

\end{figure}
\end{center}

Note that $\gamma_{p,q} = \max_{\Delta_\alpha}(px+qy) = \max (p,q\alpha)$.

Then, if $\gamma_{a,b}=a, \gamma_{c,d}=c$ then $\gamma_{a+c,b+d}=a+c$ and $\gamma_{a,b}+\gamma_{c,d}-\gamma_{a+c,b+d}=0$. So, to get a non-zero term we should have, for example,

$$\gamma_{a,b}=a, \gamma_{c,d}=d\alpha, \gamma_{a+c,b+d}=a+c, \gamma_{a,b}+\gamma_{c,d}-\gamma_{a+c,b+d}=d\alpha-c.$$

So, it follows from the geometric continued fractions algorithm (``nose stretching algorithm", a term introduced by V.I. Arnold \cite{arnold2015lectures}, who learned it from B.N. Delaunay), used to construct continued fractions geometrically, then such $(c,d)$ is a convergent $(h_n,k_n)$ for certain $n$ and the number of pairs $(a,b)$, such that $\gamma_{a,b}+\gamma_{c,d}-\gamma_{a+c,b+d}\ne 0$ is equal to the coefficient $a_{n+1}$ (and these $(c,d)$ correspond to the intermediate convergents). Thus, each term $|h_n-\alpha k_n|$ appears $a_{n+1}$ times as $\gamma_{a,b}+\gamma_{c,d}-\gamma_{a+c,b+d}$, thus proving formulas for the area $\sum  a_{n+1}(h_n-\alpha k_n)^2$ which is equal to $\alpha$ (twice the area of the triangle $BCD$ in Figure~\ref{fig_tr}).

The formulas a), c) are proved similarly: when $\alpha$ is irrational, the $SL(2,\ZZ)$ length (Definition~\ref{def_sl2zlength}) of the polyline in construction tends to zero, so $$\sum_{n=-1}^\infty a_{n+1}|h_n-\alpha k_n| = \alpha+1.$$
When $\alpha=p/q$ is rational, the process to get the formula c) finishes in a finite number of steps and the $SL(2,Z)$-length of the hypotenuse is $1/q$.
\end{proof}
\begin{remark}
The above geometric process is equivalent to the following geometric Euclidean algorithm. Start with a triangle $ABC$ such that the angle at $A$ is right. The step of the algorithm: if $|AB|>|AC|$, subtract $|AC|$ from $|AB|$ forming a new triangle $AB'C$, such that $|B'B|=|AC|$; if $|AB|\leq|AC|$, do a symmetric operation. Then apply the step of the algorithm to the new triangle. Then the areas of the carved triangles are exactly $|h_n-\alpha k_n|^2/2$ and each such triangle is carved $a_{n+1}$ times.
\end{remark}

\begin{proof}[Second proof. Induction]
We start by rational case. One can directly check identities for $\alpha = a_0$ or $\alpha=a_0+\frac{1}{a_1}$. Then we proceed by induction.

Suppose $\alpha=a_0+\frac{1}{\beta}$. Note that $\beta = [a_1,\dots, a_{N}]$. Then $$\frac{h_n(\alpha)}{k_n(\alpha)} = a_0+\frac{k_{n-1}(\beta)}{h_{n-1}(\beta)},$$ 
\[h_n(\alpha)=k_{n-1}(\beta)+ a_0h_{n-1}(\beta), k_n(\alpha)=h_{n-1}(\beta),\text{ so }\]
\[h_n(\alpha)-\alpha k_n(\alpha) = k_{n-1}(\beta)+a_0h_{n-1}(\beta)-\frac{1+\beta a_0}{\beta}h_{n-1}(\beta) = k_{n-1}(\beta)-\frac{h_{n-1}(\beta)}{\beta}\]
Thus,
$$\sum_{n=-1}^{N-1} a_{n+1}|h_n-\alpha k_n| = a_0+ \frac{1}{\beta}\sum_{n=0}^{N-1} a_{n+1}|h_{n-1}(\beta)-\beta k_{n-1}(\beta)| = a_0+\frac{1+\beta}{\beta}  = 1+\alpha.$$
The sum for squares are obtained similarly,
$$\sum_{n=-1}^{N-1} a_{n+1}|h_n-\alpha k_n|^2 = a_0+ \frac{1}{\beta^2}\sum_{n=0}^{N-1} a_{n+1}|h_{n-1}(\beta)-\beta k_{n-1}(\beta)|^2 = a_0+\frac{\beta}{\beta^2} = \alpha.$$

The irrational case follows by taking the limit of the rational case. If $\beta$ is an irrational number then the coefficients $a_n(\beta)$ and convergents $h_n(\beta),k_n(\beta)$ are the limits of the corresponding coefficients and convergents for rational numbers $\alpha$ such that $\alpha_k\to\beta$, and

$$\lim\limits_{\alpha_k\to \beta} \sum_{n=-1}^{N-1} a_{n+1}(\alpha_k)|h_n(\alpha_k)-\alpha k_n(\alpha_k)|=\sum_{n=-1}^{N-1} a_{n+1}(\beta)|h_n(\beta)-\alpha k_n(\beta)|.$$
\end{proof}

\begin{proof}[Third proof. Telescoping] We prove the identities for the irrational case, the rational case can be handled similarly.
Define the (signed) approximation error
\[
\varepsilon_{n}:=h_{n}-\alpha k_{n}, \qquad \delta_{n} = |\varepsilon_{n}|, \qquad n\ge -2.
\]

It follows from \eqref{eq_continuous} that $\varepsilon_{n+1}=a_{n+1}\varepsilon_{n}+\varepsilon_{n-1}$, thus taking into account the alternating sign of $\varepsilon_{n}$, we see that for every
$n\ge -1$,
\[
a_{n+1}\delta_{n}
   =\delta_{n-1}-\delta_{n+1}.
\]
Summing it from $n=-1$ to~$N$ we obtain
\[
\sum_{n=-1}^{N}a_{n+1}\delta_{n}
  =\delta_{-2}+\delta_{-1}-\delta_{N}-\delta_{N+1}.
\]
Because $\delta_{N}\to 0$, letting $N\to\infty$ yields
\[
\sum_{n=-1}^{\infty} a_{n+1}\delta_{n}
  =\delta_{-2}+\delta_{-1}
  =\alpha+1 .
\]
To compute the second series, multiply the relation
$\varepsilon_{n+1}-\varepsilon_{n-1}=a_{n+1}\varepsilon_{n}$ by $\varepsilon_{n}$  to get
\[
a_{n+1}\varepsilon_{n}^{2}
   =\varepsilon_{n}\varepsilon_{n+1}-\varepsilon_{n-1}\varepsilon_{n}.
\]
Summing it from $n=-1$ to~$N$:
\[
\sum_{n=-1}^{N} a_{n+1}\varepsilon_{n}^{2}
   =\varepsilon_{N}\varepsilon_{N+1}-\varepsilon_{-2}\varepsilon_{-1}.
\]
Since $\varepsilon_{N}\to 0$ we conclude
\[
\sum_{n=-1}^{\infty} a_{n+1}\varepsilon_{n}^{2}
   =-\varepsilon_{-2}\varepsilon_{-1}
   =\alpha. \qedhere
\]

\end{proof}

It is not difficult to prove similar identities for general continued fractions, or for continuous fractions over $p$-adic numbers or power series, or even multidimensional continued fractions.

\begin{remark}
The error-sum functions $$E_1(\alpha)=\sum_{n=0}^{\infty} |h_n-\alpha k_n|, \qquad E_2(\alpha)= \sum_{n=0}^{\infty} |h_n-\alpha k_n|^2$$
 were studied systematically in \cite{ridley2000error}, there are closed formulas for the Euler constant $e=2.71828\dots $ in \cite{allouche2014variations,elsner2011error}, and for quadratic irrationalities in \cite{elsner2014error}. Ahn recently transferred the question to Pierce expansions (which have different combinatorics) and obtained fractal‐dimension results \cite{ahn2023error}. In \cite{baruchel2016error}  “split-denominator” variants were introduced. In \cite{elsner2016error} by studying the integrals of the error-sums certain relations between $\pi,\log(2),\zeta(3),\zeta(5),\dots$ are discovered. Our formulas comprise the errors for the intermediate approximants.  However, we do not know how our results might be useful for that direction, besides that one can probably find explicit formulas for $$E_p(\alpha)=\sum_{n=0}^{\infty} |h_n-\alpha k_n|^p, p\geq 3$$ for quadratic irrationalities $\alpha$. It would be very intriguing to find $$\sum_{n=-1}^{\infty} a_{n+1}|h_n-e k_n|^3.$$
\end{remark}

\section{Legendre transform}
\label{sec_3}
In this section, using the Legendre transform we establish the convergence behavior of $F_\Gamma(s)$ for $s>2/3$ and reprove Theorem~\ref{thm_short2} using telescoping arguments.

Recall that a function $g^*$ is called the Legendre transformation of a $C^1$ strictly concave function $g$ if the equation of the tangent line of the slope $\alpha$ to the graph of $g$ is given by $y(x)=\alpha x - g^*(\alpha)$. 

\begin{remark}\label{rem_hata}
Observe in \eqref{eq_legendrian} that $\frac{-\gamma_{a,b}}{b} = g^*(-\frac{a}{b})$ where $g^* (\alpha), \alpha\in (-\infty,0]$ is the Legendre transform of $g$.   
Consider a $g$-triangle. Then \begin{equation}\label{eq_ggg}\gamma_{a,b}+\gamma_{c,d}-\gamma_{a+c,b+d}=(b + d)\cdot g^*\left(- \frac{a + c}{b + d} \right) - b\cdot g^*(-\frac{a}{b}) - d\cdot g^*\left(- \frac{c}{d} \right).\end{equation}
\end{remark}

\begin{lemma}\label{lemma_second}
Consider a $\Gamma$-triangle, given by a function $g:[0,1]\to\RR$, $g\in C^3([0,1])$ with $g'(0)=0, g'(1)=-1$ with $0<c_1\leq |g''(x)|<c_2$ for $x\in[0,1]$. Then, $$\gamma_{a,b}+\gamma_{c,d}-\gamma_{a+c,b+d}=\frac{(g^*)''(-\frac{a+c}{b+d})}{2bd(b+d)}+O\left(\frac{1}{{(bd(b+d))}^{3/2}}\right).$$
\end{lemma}
\begin{proof} Indeed, there exist $\xi_1\in [-\frac{a+c}{b+d},-\frac{a}{b}], \xi_2\in  [-\frac{c}{d},-\frac{a+c}{b+d}]$ such that
\[
\gamma_{a,b}+\gamma_{c,d}-\gamma_{a+c,b+d}=(b + d) g^*\left(- \frac{a + c}{b + d} \right) - b g^*(-\frac{a}{b}) - d g^*\left(- \frac{c}{d} \right)=
\]
\[
= b \left[ g^*\left(- \frac{a + c}{b + d} \right) - g^*\left(-\frac{a}{b}\right) \right]
+ d \left[ g^*\left(- \frac{a + c}{b + d} \right) - g^*\left(- \frac{c}{d} \right) \right]=
\]
\[
= b \left[ (g^*)'\left(- \frac{a + c}{b + d} \right) \cdot \frac{1}{b(b + d)}
+ \frac{1}{2} (g^*)''\left(- \frac{a + c}{b + d} \right) \cdot \left( \frac{1}{b(b + d)} \right)^2 + \frac{(g^*)'''(\xi_1)}{b^3(b+d)^3}\right]+
\]
\[
\quad +d \left[ (g^*)'\left(- \frac{a + c}{b + d} \right) \cdot \frac{-1}{d(b + d)}
+ \frac{1}{2} (g^*)''\left(- \frac{a + c}{b + d} \right) \cdot \left( \frac{1}{d(b + d)} \right)^2 + \frac{(g^*)'''(\xi_2)}{d^3(b+d)^3}\right]=
\]
\[
= \frac{1}{2} (g^*)''\left(- \frac{a + c}{b + d} \right) \cdot \frac{1}{b d (b + d)} +O\left(\frac{1}{{(bd(b+d))}^{3/2}}\right).
\]
\end{proof}

\begin{theorem}\label{thm_convergence} Consider a $\Gamma$-triangle, given by a function $g:[0,1]\to\RR$ with $g'(0)=0, g'(1)=-1$ with $0<c_1\leq |g''(x)|<c_2$ for $x\in[0,1]$. Then $F_{\Gamma}(s)$ converges for $s>2/3$ and diverges at $s=2/3$.
\end{theorem}
\begin{proof}
Since $|g''(x)|$ is bounded above and below by a positive constant, $(g^*)''$ is separated from $0$ on $[-1,0]$. It follows from Lemma~\ref{lemma_second} that the terms of $F_{\Gamma}$ differ by at most constant (depending on maximum and minimum of $(g^*)''$ on $[-1,0]$) from the corresponding terms $\frac{1}{(mn(m+n))^s}$ in $\omega(s)$ for a parabola, which has the desired behavior at $s=2/3$, see Lemma~\ref{lemma_parab}.
\end{proof}

Hence, we have the same convergence behavior for any $\Gamma$-triangle with a strictly concave curve $\Gamma$.

Next, we reprove the Theorem~\ref{thm_short2} using a telescoping argument as follows.

\begin{proof}[Proof of Theorem~\ref{thm_short2}]
Without loss of generality, consider a $\Gamma$-triangle given by a concave $g:[0,1]\to\RR$ with $g'(0)=0, g'(1)=-1$, such that the vertex of the $\Gamma$-triangle, which is not on $\Gamma$, is the origin $(0,0)$.

Start with \eqref{eq_thm11}. Note that a certain telescoping is apparent, for, using \eqref{eq_ggg}, we can define

\[S_{a+c,b+d}^{L}=\sum_{n=0}^\infty \left(\gamma_{a+c,b+d}+ \gamma_{n(a+c)+a,n(b+d)+b} - \gamma_{(n+1)(a+c)+a,(n+1)(b+d)+b}\right)=\]

\[=\lim\limits_{n\to\infty}(n\cdot\gamma_{a+c,b+d}+\gamma_{a,b}-\gamma_{n(a+c)+a,n(b+d)+b})=\]

\[=\lim\limits_{n\to\infty}\left(-n(b+d)g^*(-\frac{a+c}{b+d})-bg^*(-\frac{a}{b})+(n(b+d)+b)g^*(-\frac{n(a+c)+a}{n(b+d)+b})\right)=\]

\[=b\left(g^*(-\frac{a+c}{b+d}-g^*(-\frac{a}{b}))\right)-\frac{1}{b+d}(g^*)'(-\frac{a+c}{b+d}).\]

Similarly we define

\[S_{a+c,b+d}^{R}=\sum_{n=0}^\infty (\gamma_{a+c,b+d}\,+\, \gamma_{n(a+c)+c,n(b+d)+d} \,-\, \gamma_{(n+1)(a+c)+c,(n+1)(b+d)+d})=\]

\[=d\left(g^*(-\frac{a+c}{b+d})-g^*(-\frac{c}{d})\right)+\frac{1}{b+d}(g^*)'(-\frac{a+c}{b+d}).\]

Observe now that $S_{a+c,b+d}^{R}+S_{a+c,b+d}^{L} = \gamma_{a,b}+\gamma_{c,d}-\gamma_{a+c,b+d}$. Thus, 

\[\sum^+(\gamma_{a,b}+\gamma_{c,d}-\gamma_{a+c,b+d})=S=\sum \left(S_{a+c,b+d}^{R}+S_{a+c,b+d}^{L}\right),\]

where the last sum runs over all primitive vectors $(a+c,b+d)\in \ZZ_{>0}^2$ such that $0<a+c< b+d$. Finally, note that 
$$S=\sum \left(S_{a+c,b+d}^{R}+S_{a+c,b+d}^{L}\right) = 2S-((g^*)'(-1)-(g^*)'(0)),$$ because all interior contributions cancel and only the endpoint derivatives survive; so $S=(g^*)'(-1)-(g^*)'(0)$. Since the vertex of the $\Gamma$-triangle is the origin, a direct check shows that $(g^*)'(-1)-(g^*)'(0)$ is the $SL(2,\ZZ)$-length of two straight sides of the $\Gamma$-triangle.

To prove \eqref{eq_thm22} we need to perform telescoping for 

\[p(a,b,c,d)=bd\left(g^*(-\frac{a}{b})-g^*(-\frac{c}{d})\right)^2,\] 
because, as one can verify, 
\begin{center}
$$p(a,b,c,d)-p(a+c,b+d,c,d)-p(a,b,a+c,b+d)=-(\gamma_{a,b}+\gamma_{c,d}-\gamma_{a+c,b+d})^2,$$
\end{center}

thus 

\[\sum^+(\gamma_{a,b}+\gamma_{c,d}-\gamma_{a+c,b+d})^2 =  \lim_{n\to\infty} \sum_{\operatorname{Farey}_n} p(a,b,c,d)-p(1,1,0,1)\] where on the right we sum over all pairs of consecutive fractions $\frac{c}{d},\frac{a}{b}$ in $\operatorname{Farey}_n$. Now, 
$$\sum_{\operatorname{Farey}_n} p(a,b,c,d)=\sum_{\operatorname{Farey}_n} bd \left(g^*(-\frac{c}{d})-g^*(-\frac{a}{b})\right)^2,$$
which is a Riemann sum for $\int\limits_{-1}^0 ((g^*)')^2$, because $\frac{1}{bd}$ is the length of the interval $[-\frac{a}{b},-\frac{c}{d}]$, and these intervals from $\operatorname{Farey}_n$ partition $[-1,0]$.

Then, $p(1,1,0,1)=(g^*(-1)-g^*(0))^2$, so 

\begin{equation}
\label{eq_gg}
\sum^+(\gamma_{a,b}+\gamma_{c,d}-\gamma_{a+c,b+d})^2=-\big(g^*(-1)-g^*(0)\big)^2+\int_{-1}^0 ((g^*)')^2,
\end{equation}
which equals twice the area of our $\Gamma$-triangle. Note that in the above proof we assumed that $g^*$ is $C^1([-1,0])$. If it is not so, we may approximate $g$ by strictly concave $C^2([0,1])$ functions, and then take the limit in all formulas.

\end{proof}

\section{Connections to Hata's work}
\label{sec_mixed}

In \cite{hata1995farey}, Masayoshi Hata had developed a method to derive summation identities involving Farey fractions using a piecewise linear Schauder basis. His approach leads to several classical and new identities, including an interesting formula for Euler’s constant \( \gamma \) that we discuss below. The construction uses Farey pairs and associated fundamental intervals. Interestingly, his approach is equivalent to ours, up to Legendre transform. Below we recall Hata's definitions and results and show how to rephrase them in our language.

To get the notation straight, suppose that we consider a $\Gamma$-triangle such that the tangent at $\Gamma(0)$ has slope $1$ and the tangent at $\Gamma(1)$ is horizontal. Thus, the possible slopes of tangents belong to $[0,1]$.

{\bf Farey intervals and Schauder bases.} Let $F$ be the set of Farey pairs in $[0,1]$. 
Each \( I = [c/d, a/b] \in F \) is called a (fundamental) {\it Farey interval}, $ad-bc=1$.

\begin{remark}
Note that elements of $F$ bijectively correspond to elements in $SL_+(2,\ZZ)$ with $b\geq a$ and $d\geq c$.
\end{remark}

To each interval \( I = [c/d, a/b] \), Hata associates a Schauder base function:
\[
S_I(x) = \frac{b + d}{2} \left( |a - bx| + |c - dx| - |a + c - (b + d)x| \right).
\]
The function \( S_I \) is supported on \( I \), continuous, piece-wise linear, zero outside $I$, and has unit $L_\infty$ norm. The collection \( \{S_I\}_{I \in F} \) is a Schauder basis for \( C([0, 1]) \), in the sense that every \( f \in C([0,1]) \) admits a unique expansion:
\[
f(x) = f(0) + (f(1) - f(0))x + \sum_{I \in F} c_I(f) S_I(x).
\]
The coefficients \( c_I(f) \) can be explicitly found:
\[
c_I(f) = f\left(\frac{a + c}{b + d}\right) - \frac{b}{b + d} f\left(\frac{a}{b}\right) - \frac{d}{b + d} f\left(\frac{c}{d}\right).
\]

\begin{remark}
These functions $S_I(s)$ were known also to H. Montgomery and J. Hubbard and play a major role in a thesis \cite{haynes2006tools} about metric number theory, where, in particular, a number of results were derived from the fact that the derivatives of $S_I$ form a complete and orthogonal basis and a complete system of martingale differences for $L^2([0,1])$, \cite{haynes2008martingale}.

\end{remark}

\begin{figure}[h]
\begin{tikzpicture}[scale=6]

\draw[->] (-0.02, 0) -- (1.05, 0) node[below] {};
\draw[->] (0, -0.05) -- (0, 0.7) node[left] {};

\draw[thick, domain=0:1, samples=200, smooth, variable=\x]
  plot ({\x}, {0.6 * sin(180 * \x^(0.8))});

\coordinate (A) at (0,0);
\coordinate (B) at ({1/3}, {0.6 * sin(180 * (1/3)^0.8)});
\coordinate (C) at ({2/5}, {0.6 * sin(180 * (2/5)^0.8)});
\coordinate (D) at ({1/2}, {0.6 * sin(180 * (1/2)^0.8)});
\coordinate (E) at (1,0);

\foreach \x in {1/3, 2/5, 1/2} {
  \draw[dashed] ({\x},0) -- ({\x},{0.6 * sin(180 * \x^(0.8))});
}

\draw (A) -- (D) -- (E) -- cycle;  
\draw (A) -- (B) -- (D);           
\draw (B) -- (C) -- (D);           

\foreach \point in {A, B, C, D, E} {
  \fill (\point) circle (0.3pt);
}

\node[above left] at (A) {A};
\node[above] at (B) {B};
\node[above] at (C) {C};
\node[above] at (D) {D};
\node[above right] at (E) {E};

\node[below] at (0.03, 0) {\small 0}; 
\node[below] at (1/3, 0) {\small $\frac{1}{3}$};
\node[below] at (2/5, 0) {\small $\frac{2}{5}$};
\node[below] at (1/2, 0) {\small $\frac{1}{2}$};
\node[below] at (1, 0) {\small 1};

\end{tikzpicture}
\caption{The first few steps of constructing $c_I$ for $I=[0,1],[0,1/2], [1/3,1/2]$. }
\label{fig_last}
\end{figure}

{\bf Parseval-type identity.} One of Hata's key results is a Parseval-type identity for these coefficients
\cite[Corollary 3.4]{hata1995farey}: For any \( f \in C^2([0, 1]) \),
\begin{equation}\label{eq_hataarea}
\sum_{I \in F} (b + d)^2 c_I(f)^2 = \int_0^1 (f'(x))^2\,dx- (f(1) - f(0))^2  \Leftrightarrow \text{Theorem~\ref{thm_short}}, \eqref{eq_thm2}. 
\end{equation}

This result mirrors the classical Parseval identity in Fourier analysis. 

Following Hata's work \cite{hata1995farey}, our formulas can be interpreted as Parseval-type identities, since, by Remark~\ref{rem_hata}, if $I={[\frac{c}{d},\frac{a}{b}]}$ and $f=g^*$, then
$$\gamma_{a,b} + \gamma_{c,d} - \gamma_{a+c,b+d}=(b+d)c_I(f).$$  
Hence the above formula \eqref{eq_hataarea} is equivalent to the part \eqref{eq_thm2} in Theorem~\ref{thm_short} for the area and we proved it in Section~\ref{sec_3} as \eqref{eq_gg}.

Consider \cite[Theorem 5.3]{hata1995farey}. For any $f\in C^2[0,1]$ we have
$$\sum_{I \in F} (b+d)c_I(f) = f'(0)-f'(1)  \Leftrightarrow \text{Theorem~\ref{thm_short}}, \eqref{eq_thm1}.$$
Due to Remark~\ref{rem_hata} this  formula is equivalent to the part \eqref{eq_thm1} in Theorem~\ref{thm_short} for the $SL(2,\ZZ)$-perimeter (and we proved it in Section~\ref{sec_3}.). Note that the sum does not depend on the behavior of $f$ inside $[0,1]$ once we fixed $f'(0),f'(1)$, similarly \eqref{eq_thm1} depends only on two straight sides of a $g$-triangle.

{\bf Application to Euler’s constant.} Another useful identity is \cite[Theorem 3.3]{hata1995farey}:  For any $f\in C([0,1])$, $g\in C^2([0,1])$ we have

\begin{multline}\label{eq_mixedhata}
\sum_{I \in F} (b+d)^2c_I(f)c_I(g) = \\ 
=f(1)g'(1)-f(0)g'(0)-(f(1)-f(0))(g(1)-g(0))-\int_0^1f(x)g''(x)\,dx.
\end{multline}

It has an immediate Corollary \cite[Corollary 3.5]{hata1995farey}: For any \( f \in C([0,1]) \),
\begin{equation}\label{eq_g}
\sum_{I \in F} \frac{c_I(f)}{bd} = 2 \int_0^1 f(x)\,dx - f(0) - f(1).
\end{equation}
This can be seen in Figure~\ref{fig_last} (where $f(0)=f(1)=0$) for the area of each small triangle is equal to $\frac{c_I}{bd}$, and the sum of the areas is the integral (the area under the curve).
 Let
\[
\psi(x) = x \left\{\frac{1}{x}\right\} \left(1 - \left\{\frac{1}{x}\right\}\right),
\]
where \( \{x\} \) denotes the fractional part of \( x \). Hata shows that for all Farey intervals \( I = [c/d, a/b] \) except intervals $[0,1/n]$, one has:
\[
c_I(\psi) = \frac{1}{a c (a + c)(b + d)}.
\]
It turns out that:
\[
\int_0^1 \psi(x)\,dx = \gamma - \frac{1}{2},
\]
where \( \gamma \) is Euler’s constant. We now apply \eqref{eq_g} with $f=\psi$, obtaining the formula:
\[
2 \int_{1/{n+1}}^{1/n} \psi(x)\,dx = \sum_{I \in F, I\subset [\frac{1}{n+1},\frac{1}{n}]} \frac{c_I(\psi)}{bd} - \psi(\frac{1}{n}) - \psi(\frac{1}{n+1}).
\]
Since \( \psi(0) = \psi(1) = \psi(\frac{1}{n})= 0 \), after summatioin this becomes:
\[
2(\gamma - 1/2) = \sum_{I \in F\setminus \{[\frac{1}{n},\frac{1}{n+1}], n=1,2,\dots\}} \frac{1}{abcd(a+c)(b+d)}.
\]
Thus:
\begin{theorem}[\cite{hata1995farey}, Theorem 4.1]
Euler's constant \( \gamma \) can be expressed as
\begin{equation}\label{eq_hata}
\gamma = \frac{1}{2} + \frac{1}{2} \sum_{\substack{b,d \geq 1 \\ \gcd(b,d) = 1 \\ b \geq 2}} \frac{1}{abcd(a + c)(b + d)},
\end{equation}
where the sum ranges over Farey intervals \( I = [c/d, a/b] \in F \) with \(b\geq 2 \), i.e. over all Farey intervals except $[0,1/n], n=1,2,\dots$
\end{theorem}

We derive a short telescopic proof of the above formula in \cite{kalinineuler} and derive it from more general theorems about summing over topographs in \cite{kalinintopo}.

{\bf Identities arising from mixed areas of triangles.} The identities \eqref{eq_mixedhata},\eqref{eq_g} can be reinterpreted as mixed areas using the following theorem, inspired by Hata’s framework:

\begin{theorem}
Let $g_1, g_2$ be two concave functions as in Section~\ref{sec_one}, and let $\gamma_{a,b}$ and $\delta_{a,b}$ be the coefficients of the supporting lines $ax + by = \gamma_{a,b}$ and $ax + by = \delta_{a,b}$ for the graphs of $g_1$ and $g_2$, respectively. Then,
\[
\sum^{+} (\gamma_{a,b} + \gamma_{c,d} - \gamma_{a+c,b+d})(\delta_{a,b} + \delta_{c,d} - \delta_{a+c,b+d})
\]
is equal to the mixed area of the $g_1$-triangle and the $g_2$-triangle.
\end{theorem}

\begin{proof}

Let $X$ denote the Minkowski sum of the $g_1$-triangle and the scaled $g_2$-triangle, $\epsilon \cdot g_2$. Then $X$ is again a $g$-triangle for some function $g$, whose support function in direction $(a, b)$ is given by
\[
\beta_{a, b} = \gamma_{a, b} + \epsilon \delta_{a, b}.
\]
Hence, applying the area formula from Theorem~\ref{thm_short}, we compute:
\[
\mathrm{Area}(X) = \frac{1}{2} \sum^{+} (\beta_{a, b} + \beta_{c, d} - \beta_{a+c, b+d})^2.
\]
Substituting $\beta_{a, b} = \gamma_{a, b} + \epsilon \delta_{a, b}$, we obtain:
\begin{align*}
\mathrm{Area}(X) &= \frac{1}{2} \sum^{+} \left( \gamma_{a, b} + \gamma_{c, d} - \gamma_{a+c, b+d} + \epsilon (\delta_{a, b} + \delta_{c, d} - \delta_{a+c, b+d}) \right)^2 = \\
&= \frac{1}{2} \sum^{+} \left[ (\gamma_{a, b} + \gamma_{c, d} - \gamma_{a+c, b+d})^2 + \epsilon^2 (\delta_{a, b} + \delta_{c, d} - \delta_{a+c, b+d})^2 \right. \\
&\qquad\left. + 2\epsilon (\gamma_{a, b} + \gamma_{c, d} - \gamma_{a+c, b+d})(\delta_{a, b} + \delta_{c, d} - \delta_{a+c, b+d}) \right].
\end{align*}
Thus, the coefficient of $\epsilon$ in this expression is
\[
\sum^{+} (\gamma_{a, b} + \gamma_{c, d} - \gamma_{a+c, b+d})(\delta_{a, b} + \delta_{c, d} - \delta_{a+c, b+d}),
\]
which equals the mixed area of the $g_1$-triangle and the $g_2$-triangle, as desired.

\end{proof}

\begin{remark}
Hata's formula \eqref{eq_hata} is obtained when we take two parabolas $y=x^2$ and $x=y^2$. One considers multiple $\Gamma$-triangles which curvilinear sides are segments of parabolas corresponding to slopes in intervals $[\frac{1}{n+1},\frac{1}{n}], n=1,2, \dots$.
\end{remark}

\section{Discussion}
\label{sec_disc}

{\bf Reduced matrices, Selberg trace formula}. Beyond the well-known links with Farey and continued fractions, a third classical
viewpoint is \emph{reduction theory} in $SL(2, \mathbb{Z})$. 
Let
\[
SL_{+}(2,\mathbb{Z})\;=\;\Bigl\{
\begin{pmatrix} a & b \\ c & d \end{pmatrix}\in SL(2, \mathbb{Z})
:\ a,b,c,d\ge 0\Bigr\}
\]
be the ``positive part’’ of $SL(2,\ZZ)$. Right multiplication by the
shift matrix
\[
T=\begin{pmatrix}1&1\\[1pt]0&1\end{pmatrix}
\]
gives a natural bijection between $SL_{+}(2,\mathbb{Z})$ and the set of
\emph{reduced matrices}
\[
\mathrm{Red}\;=\;\Bigl\{
\begin{pmatrix} a & b \\ c & d \end{pmatrix}\in SL(2, \mathbb{Z})
:\ b\ge a\ge 0,\ d\ge c\ge 0\Bigr\}.
\]
Reduced matrices correspond to Farey pairs and appear in this role in Theorem~\ref{thm_convergence}. Two key (classical) properties of reduced matrices are:

\begin{enumerate}
\item (\emph{Unique factorisation / continued fractions}) Every
$\gamma\in\mathrm{Red}$ is uniquely a product
\[
\gamma=\begin{pmatrix}0&1\\[1pt]1&n_1\end{pmatrix}
\begin{pmatrix}0&1\\[1pt]1&n_2\end{pmatrix}
\cdots
\begin{pmatrix}0&1\\[1pt]1&n_k\end{pmatrix},
\qquad n_i\in\mathbb{Z}_{>0},
\]
which encodes the continued fraction data.\vspace{2pt}
\item (\emph{Reduction of hyperbolic elements}) Every hyperbolic
$\gamma\in SL(2, \mathbb{Z})$ is conjugate (up to sign) to at least one and only
finitely many elements of $\mathrm{Red}$.
\end{enumerate}

In \cite{LewisZagier1996,Mayer1991} the Selberg zeta function is expressed as
\(
Z_\Gamma(s)=\det(1-L_s)
\),
where \(L_s\) is Mayer’s transfer operator acting on a Banach space \(V\) of
holomorphic functions (e.g.\ functions holomorphic on a certain disk \(D\) with the maps
\(z\mapsto (az+b)/(cz+d)\) preserving \(D\)). Concretely,
\[
(L_s f)(z)
 \;=\;
 \sum_{n=1}^{\infty} (z+n)^{-2s}\,
 f\!\left(\frac{1}{\,z+n\,}\right),
 \qquad \Re(s)>\tfrac12.
\]
Equivalently, if \(\pi_s\) denotes the (right) action
\[
(\pi_s(\gamma)f)(z)= (cz+d)^{-2s}\,f\!\left(\frac{az+b}{cz+d}\right),
\qquad
\gamma=\begin{pmatrix}a&b\\ c&d\end{pmatrix}\in SL(2, \mathbb{Z}),
\]
and if
\[
A=\begin{pmatrix}1&1\\[1pt]0&1\end{pmatrix},
\qquad
P=\begin{pmatrix}0&1\\[1pt]1&1\end{pmatrix},
\qquad
L \;=\; (1-[A])^{-1}[P]
 \;=\; \sum_{n=1}^{\infty} \Bigl[\begin{pmatrix}1&n\\[1pt]0&1\end{pmatrix}\Bigr]
\]
in the group ring, then \(L_s=\pi_s(L)\) and \(Z_\Gamma(s)=\det(1-L_s)\). This identity
interfaces with the Selberg trace formula by identifying the logarithmic derivative
of \(Z_\Gamma(s)\) with both the spectral sum over Laplace eigenvalues and the
geometric sum over the length spectrum of closed geodesics on \(\mathbb{H}/{GL(2,\ZZ)}\).

In our article, the 
``positive part’’ of $SL(2,\ZZ)$ is used to index summands. Under the above bijection,
\emph{reduced matrices} correspond canonically to the $\Gamma$-triangles
appearing in  Section~\ref{sec_3}. Moreover, the transfer
operators in \cite{LewisZagier1996,Mayer1991}, e.g.\ Mayer’s operator
$L_s f(z)=\sum_{n\ge1}(z+n)^{-2s}f(\tfrac{1}{z+n})$, are structurally close to
our telescoping terms $S^L,S^R$ for $s=1$ (see Section~\ref{sec_3}). Perhaps, interpretation of our function $F_\Gamma(s)$ for $s\ne 1,2$ could be obtained if the connection of $F_\Gamma(s)$ to Selberg trace formula is clarified.

{\bf  Abelianisation and a $12$-fold decomposition.} The abelianisation of $SL(2, \mathbb{Z})$ is cyclic of order $12$:
\[
SL(2, \mathbb{Z})^{\mathrm{ab}}
\;\cong\; SL(2, \mathbb{Z})/[SL(2, \mathbb{Z}),SL(2, \mathbb{Z})]
\;\cong\; \mathbb{Z}/12\mathbb{Z}.
\]
Composing our indexing of summands by $\gamma\in SL_+(2,\mathbb{Z})$ with the
abelianisation map $\,\mathrm{ab}:SL(2, \mathbb{Z})\to\mathbb{Z}/12\mathbb{Z}\,$
splits any global sum into a direct sum of $12$ arithmetic pieces (cf. Definition~\ref{def_zeta}):
\[
F_g(s)
\;=\;\sum_{r\in\mathbb{Z}/12\mathbb{Z}}\;
F_g^{(r)}(s),
\qquad
F_g^{(r)}(s)
\;=\;\!\!\sum_{\substack{\gamma\in SL_+(2,\mathbb{Z})\\ \mathrm{ab}(\gamma)=r}}
\!\!f_g(\gamma)^s.
\]
It would be interesting to study them, for such a $12$-fold splitting can reveal
finer cancellation patterns and symmetry. It may also interact cleanly with continued-fraction
factorisations, since the image of $\gamma$ in $\mathbb{Z}/12\mathbb{Z}$ is
additive under concatenation of the elementary factors
$\begin{pmatrix}0&1\\[1pt]1&n\end{pmatrix}$.

{\bf Finite-field analogues and supercongruences.}

A natural direction is to pass to finite fields. Replace $\mathbb{Z}$ by $\mathbb{F}_q$ (with $q=p^k$) and consider the action of $SL(2, \mathbb{F}_q)$ on the projective line $P^1(\mathbb{F}_q)$.
Given homogeneous classes $(a\!:\!b),(c\!:\!d)\in P^1(\mathbb{F}_q)$ with $ad-bc=1$ in $\mathbb{F}_q$, one can view the ordered pair as an oriented $1$-simplex. 
This suggests seeking a character–sum or trace–operator avatar over $\mathbb{F}_q$ of the geometric “telescoping” we exploit over $\ZZ$. Maybe similar telescoping exists for  the spherical Tits building for $SL(2,\mathbb Q_p)$ for the case of $p$-adic numbers.

When the archimedean identities admit $p$-adic lifts, e.g. the Gross–Koblitz formula expressing Gauss sum via $p$-adic gamma function, one expects \emph{supercongruences} for the finite-field counterparts. It is natural to ask whether Farey--Legendre telescoping, after a suitable mod--$p$ interpretation, yields families of sums whose values modulo $p^k$ mirror the archimedean identities to high $p$-adic precision.

I am grateful to the University of Geneva, where this work was carried out, I thank Misha Shkolnikov and Ernesto Lupercio for inspiring discussions and an anonymous referee for comments.


\end{document}